\documentclass[10pt]{amsart}

\usepackage{amsmath} \usepackage{amssymb, amscd,cancel, graphicx,soul}
	\usepackage{pinlabel,mathtools}

\usepackage{hyperref}
\hypersetup{
    colorlinks=true, 
    linktoc=all,     
    linkcolor=black,  
}

\usepackage{enumerate}

\usepackage{cleveref}
\usepackage{appendix}
\usepackage{tikz-cd}

\usepackage{mathdots}
\newtheorem{thm}{Theorem}[section]

\newtheorem{conj}[thm]{Conjecture}
\newtheorem{cor}[thm]{Corollary}
\newtheorem{lem}[thm]{Lemma}
\newtheorem{prop}[thm]{Proposition}

\newtheorem*{ques*}{Question}
\newtheorem{defin}[thm]{Definition}

\newtheorem*{example*}{Example}
\newtheorem*{porism*}{Porism}
\newtheorem*{scholium*}{Scholium}

\newtheorem*{thm*}{Theorem}
\newtheorem*{defin*}{Definition}
\newtheorem*{lem*}{Lemma}
\newtheorem*{prop*}{Proposition}
\newtheorem*{remark*}{Remark}

\def\cB{{\mathcal B}}
\def\cC{{\mathcal C}}

\def\cJ{{\mathcal J}}
\def\cM{{\mathcal M}}

\def\bC{{\mathbb C}}
\def\C{{\mathbb C}}

\def\bF{{\mathbb F}}

\def\M{{\mathrm M}}

\def\bD{{\mathbb D}}

\def\bR{{\mathbb R}}
\def\T{{\mathbb T}}
\def\bZ{{\mathbb Z}}

\def\Crit{{\mathrm{Crit}}}

\def\del{{\partial}}

\def\im{{\textup{im}}}

\def\ev{{\mathrm{ev}}}
\def\omegast{{\omega_\mathrm{std}}}

\newcommand{\into}{\hookrightarrow}

\newcommand{\ol}{\overline}

\begin{document}
\thispagestyle{empty}
\title[Polynomial inscriptions]{Polynomial inscriptions}
\author{Joshua Evan Greene} 
\address{Department of Mathematics, Boston College, USA}
\email{joshua.greene@bc.edu}
\urladdr{https://sites.google.com/bc.edu/joshua-e-greene}
\author{Andrew Lobb} 
\address{Mathematical Sciences,
	Durham University,
	UK}
\email{andrew.lobb@durham.ac.uk}
\urladdr{http://www.maths.dur.ac.uk/users/andrew.lobb/}
\thanks{JEG was supported by the National Science Foundation under Award No.~DMS-2304856 and by a Simons Fellowship.}

\begin{abstract}
We prove that for every smooth Jordan curve $\gamma \subset \bC$ and for every set $Q \subset \bC$ of six concyclic points, there exists a non-constant quadratic polynomial $p \in \bC[z]$ such that $p(Q) \subset \gamma$.
The proof relies on a theorem of Fukaya and Irie.
We also prove that if $Q$ is the union of the vertex sets of two concyclic regular $n$-gons, there exists a non-constant polynomial $p \in \bC[z]$ of degree at most $n-1$ such that $p(Q) \subset \gamma$.
The proof is based on a computation in Floer homology.
These results support a conjecture about which point sets $Q \subset \bC$ admit a polynomial inscription of a given degree into every smooth Jordan curve $\gamma$.
\end{abstract}

\maketitle

\section{Introduction.}
\label{sec:intro}
In a sequence of papers, we introduced symplectic geometry into the study of inscription problems in the plane.
Given a finite set of points $Q$, a Jordan curve $\gamma$, and a set $\M$ of continuous maps of the plane $\bR^2 \rightarrow \bR^2$, the inscription problem asks whether there exists a map $f \in \M$ such that $f(Q) \subset \gamma$.
The prototype is the well-known and unsolved {\em Square Peg Problem} due to Toeplitz \cite{toeplitz1911}: in this case, $Q$ is the vertex set of a square and $\M = \mathrm{Sim}^+(\bR^2)$ is the set of orientation-preserving similarities.
Another unsolved (although see \cite{vreziv}) variant is due to Gr\"unbaum \cite{grunbaum}: in this case, $Q$ is the vertex set of a regular hexagon and $\M = \mathrm{Aff}(\bR^2)$ is the set of affine transformations.

The relevance of symplectic geometry stems from the observation that some inscription problems can be recast as problems about Lagrangian intersections.
For example, we proved that if $Q \subset \bR^2$ is a set of four concyclic points and $\gamma \subset \bR^2$ is a smooth Jordan curve, then there exists $f \in \mathrm{Sim}^+(\bR^2)$ such that $f(Q) \subset \gamma$ \cite{greenelobb2}.  (Points are called \emph{concyclic} if they lie on a common circle, while a set of concyclic points is called \emph{cyclic}.)
The proof is based on the result of Polterovich and Viterbo that a Lagrangian embedding of the torus $\mathbb{T}^2$ in the symplectic vector space $(\bC^2,\omega_\mathrm{std})$ contains a loop of Maslov index 2
\cite{polto1991,viterbo1990}.
The result is also optimal in the sense that if $Q \subset \bR^2$ 
is not cyclic, and $\gamma$ is a circle, then there does not exist $f \in \mathrm{Sim}^+(\bR^2)$ such that $f(Q) \subset \gamma$.

\subsection*{Polynomials.}
Under the identification $\bR^2 = \bC$, the set $\mathrm{Sim}^+(\bR^2)$ is identified with the set of degree-1 polynomials with complex coefficients.
Thus, the preceding result asserts that every set $Q \subset \bC$ of four concyclic points admits a {\em linear inscription} into every smooth Jordan curve $\gamma \subset \bC$.
This perspective leads us to consider inscriptions by polynomials of higher degree.
Let $d \ge 1$, and consider the space of non-constant complex polynomials of degree $\le d$:
\[
\bC[z]_d^* := \bC[z]_d \setminus \bC {\rm .}
\]
We say that a set $Q \subset \bC$ admits a {\em degree-$d$ inscription into} $\gamma$ if there exists $p \in \bC[z]_d^*$ such that $p(Q) \subset \gamma$, and we write
\[
I_d^*(Q,\gamma) := \{p \in \C[z]_d^* : p(Q) \subset \gamma\}
\]
for the {\em moduli space} of degree-$d$ inscriptions of $Q$ into $\gamma$.
We drop the $*$ if we wish to include constant polynomials.
A parameter-counting heuristic below shows that the expected dimension of this space is 0 when $|Q| = 2(d+1)$.
In this case, $I_d^*(Q,\gamma)$ generically consists of a finite number of degree-$d$ inscriptions.
Thus, our motivating problem is:

\begin{center}
	\fbox{\begin{minipage}{29em}
			For which sets $Q \subset \bC$ of $2(d+1)$ points it is the case that $Q$ admits a degree-$d$ inscription into every smooth Jordan curve $\gamma \subset \bC$?
	\end{minipage}}
\end{center}
\Cref{conj:reducible2} below gives our conjectured answer, and we build up to it gradually.

In line with our earlier result in the case $d=1$, we first pose the following:

\begin{conj}
\label{conj:cyc_poly_peg_prob}
If $Q \subset \bC$ is a set of $2(d+1)$ concyclic points and $\gamma \subset \bC$ is a smooth Jordan curve, then there exists a degree-$d$ inscription of $Q$ in $\gamma$.
\end{conj}
\noindent
Our main result confirms \Cref{conj:cyc_poly_peg_prob} in the case $d=2$:

\begin{thm}
\label{thm:quadratic_inscriptions}
If $Q \subset \bC$ is a set of six concyclic points and $\gamma \subset \bC$ is a smooth Jordan curve, then there exists a quadratic inscription of $Q$ into $\gamma$.
\end{thm}

The symplectic method applies as follows.
Suppose that $Q \subset \bC$ consists of $2(d+1)$ points.
We show that $I_d(Q,\gamma)$ is parametrized by the set of intersections between a pair of $(d+1)$-dimensional tori in $\bC^{d+1}$.
Remarkably, when the points of $Q$ are concyclic, the tori are Lagrangian in the symplectic vector space $(\bC^{d+1},\omegast)$, and the constant inscriptions form a clean loop of intersection between them.
By smoothing away this loop, we see that $I_d^*(Q,\gamma)$ is parametrized by the set of self-intersections of a Lagrangian immersion of the manifold $S^1 \times (\mathbb{T}^d \# \mathbb{T}^d)$ into $(\bC^{d+1},\omegast)$.
This immersion has the property that the circle fiber $S^1 \times \{ {\rm pt} \}$ has Maslov index $2(d+1)$, generalizing the fact that when $d=1$, it is a Lagrangian immersion of the torus $\mathbb{T}^2$ with minimum Maslov number 4.
We pose the following:

\begin{conj}
\label{conj:maslov2}
In every Lagrangian embedding $S^1 \times (\mathbb{T}^d \# \mathbb{T}^d) \into (\C^{d+1},\omega_\mathrm{std})$, $d \ge 2$, the circle fiber has Maslov index 2.
\end{conj}
\noindent
We discuss \Cref{conj:maslov2} in the subsection on Lagrangian embeddings.
Arguing as above, we shall establish the following:
\begin{prop}
	\label{prop:conj_implies_conj}
\Cref{conj:maslov2} implies \Cref{conj:cyc_poly_peg_prob}.
\end{prop}
	
\begin{proof}[Proof of \Cref{thm:quadratic_inscriptions}.]
Fukaya sketched a proof that if $L$ is a compact, relatively spin, displaceable, aspherical Lagrangian in a symplectic manifold $(M,\omega)$, then there exists a loop in $L$ with Maslov index 2 whose centralizer has finite index in $\pi_1(L)$ \cite[Theorem 12.1]{fukaya_langrangian_submanifolds}.
Irie put the proof onto firm footing by resolving a transversality issue left open by Fukaya  \cite{irie}.
The hypotheses of the Fukaya-Irie theorem hold for a Lagrangian embedding $S^1 \times (\mathbb{T}^2 \# \mathbb{T}^2) \approx S^1 \times \Sigma_2 \into (\bC^3,\omega_\mathrm{std})$.
A loop with finite index in $\pi_1(S^1 \times \Sigma_2)$ is homotopic to a multiple of the circle fiber $S^1 \times \{ {\rm pt} \}$, and since $S^1 \times \Sigma_2$ is orientable, every class has even Maslov index.
It follows that the loop guaranteed by the Fukaya-Irie theorem is homotopic to the circle fiber with some orientation.
This confirms the case $d=2$ of \Cref{conj:maslov2}, so the case $d=2$ of \Cref{prop:conj_implies_conj} yields the desired result.
\end{proof}

As further evidence for \Cref{conj:cyc_poly_peg_prob}, define a $2(d+1)$-{\em pinwheel} to be the union of the set of vertices of a regular $(d+1)$-gon with its image under a rotation through some angle about the center of the polygon.

\begin{thm}
\label{prop:pinwheel}
A $2(d+1)$-pinwheel admits a degree-$d$ inscription into every smooth Jordan curve.
\end{thm}
\noindent
The proof of \Cref{prop:pinwheel} uses Floer theory.
In this case, the parametrizing Lagrangian tori are monotone and Hamiltonian isotopic, hence their Floer homology is unobstructed.
They also admit a symmetry which we are able to exploit in order to compute their pearl Floer chain complex \cite{schmaschke}.
The case $d=1$ gives a new argument, albeit a variation on a theme, that rectangles admit linear inscriptions into smooth Jordan curves.
Using this pair of tori and more Floer theoretic methods in the case $d=1$, we were able to prove the existence of linear inscriptions of an intervals' worth of rectangles into all {\em rectifiable} Jordan curves \cite[Theorem A]{greenelobb4}.
Although we do not study polynomial inscriptions of point sets into non-smooth Jordan curves here, pinwheels stand out as a natural candidate for further study.

\subsection*{Optimality.}
By contrast with the result about linear inscriptions of cyclic quadrilaterals, \Cref{thm:quadratic_inscriptions} does \emph{not} account for all six-point sets which admit quadratic inscriptions into all smooth Jordan curves.
Indeed, fix a set $P \subset \bC$ of four concyclic points and a point $c \in \bC$, and let $Q$ consist of six elements of $c+ \sqrt{P} = \{ z \in \bC : (z - c)^2 \in P\}$.
It is easy to arrange that the points of $Q$ are not concyclic.
On the other hand, if $\gamma$ is a smooth Jordan curve, then there exists a linear inscription $az+b$ of $P$ into $\gamma$, so $p(z) = a(z - c)^2 + b$ is a quadratic inscription of $Q$ into $\gamma$.
We call such a point set $Q$ {\em cyclically reducible}.

On the other hand, not every six-point set admits a quadratic inscription into every smooth Jordan curve.
For instance, if $Q$ consists of six {\em colinear} points, then there does not exist a quadratic inscription of $Q$ into a circle.
This may be argued from B\'ezout's theorem: the image of a line containing $Q$ under a non-constant quadratic polynomial defines a degree-2 algebraic curve in the plane, which intersects the circle (another such curve) in no more than four points with multiplicity.
More generally, six points quadratically inscribe into a circle if and only if they lie on a {\em Cassini oval}\footnote{\emph{``Cette ligne est une maniere d'ellipse dans laquelle les rectangles faits par les lignes tir{\'e}es de la planette {\`a} l'un et {\`a} l'autre foyer font to{\^u}jours {\'e}gaux$\ldots$''} \cite{cassini}.}, although we do not know a versatile characterization of this property.

Complementing \Cref{thm:quadratic_inscriptions}, we pose the following:

\begin{conj}
\label{conj:reducible}
A set $Q \subset \bC$ of six points admits a quadratic inscription into every smooth Jordan curve $\gamma \subset \bC$ if and only if $Q$ is either cyclic or cyclically reducible.
\end{conj}

More generally, call a set $Q \subset \bC$ of $2(d+1)$ points {\em cyclically reducible} if there exist integers $a \ge 2$ and $b \ge 1$ and a polynomial $p \in \bC[z]^*_a$ such that $ab \le d$ and $p(Q)$ is a set of $\le 2(b+1)$ concyclic points.
\Cref{conj:cyc_poly_peg_prob} implies in this case that $Q$ admits a degree-$d$ inscription into every smooth Jordan curve.
Most generally, we pose:

\begin{conj}
\label{conj:reducible2}
A set $Q \subset \bC$ of $2(d+1)$ points admits a degree-$d$ inscription into every smooth Jordan curve $\gamma \subset \bC$ if and only if $Q$ is either cyclic or cyclically reducible.
\end{conj}

\noindent
An affirmative answer to \Cref{conj:reducible2} would therefore solve our motivating problem.

\subsection*{Lagrangian embeddings.}
We give a little more context for \Cref{conj:maslov2}.
Audin conjectured that every Lagrangian embedding of the $n$-torus $\mathbb{T}^n$ in the symplectic vector space $(\bC^n,\omega_\mathrm{std})$ contains a loop of Maslov index 2 \cite{audin1988}.
This conjecture stimulated a lot of work, beginning with the result of Polterovich and Viterbo, and continuing beyond its solution by Cieliebak-Mohnke \cite[Theorem 1.2(b)]{cieliebakmohnke2018}.
The state-of-the-art result is due to Fukaya and Irie, who proved that every Lagrangian embedding of a closed {\em aspherical} manifold (i.e. a $K(\pi,1)$ space) in a symplectic vector space contains a loop of Maslov index 2.
On the other hand, there exist closed Lagrangian submanifolds in a symplectic vector space without a loop of Maslov index 2: for example,
$S^1 \times S^2$ admits a Lagrangian embedding in $\bC^3$ of vanishing Maslov class \cite[Corollary 1.6]{eems2013}.
Beyond the condition of asphericity, there is not yet a good conjectural picture tying the topology of Lagrangian submanifolds $L \subset \bC^n$ to the Maslov 
class $\mu \in H^1(L;\bZ)$.
Following Fukaya, such a picture should involve the string topology of the free loop space of $L$, and the case of a connected sum of tori is not yet well understood.

\subsection*{Parameter counting.}
The following heuristic justifies studying the case $|Q| = 2(d+1)$; it also offers a glimpse as to how the parametrizing tori appear.
Every set $Q$ of $d+1$ points admits a degree-$d$ inscription into every Jordan curve $\gamma$.
Indeed, if $Q = \{q_1,\dots,q_{d+1}\}$ and we select arbitrary targets $z_1,\dots,z_{d+1} \in \gamma$, then there exists a unique polynomial $p \in \bC[z]_d$ such that $p(q_i) = z_i$, $i=1,\dots,d+1$.
It follows that, including the constant inscriptions, the moduli space
$I_d(Q,\gamma)$
is a $(d+1)$-dimensional torus $\mathbb{T}^{d+1}$, parametrized by the images $(p(q_1),\dots,p(q_{d+1})) \in \gamma^{d+1}$.
If instead $Q = \{q_1,\dots,q_k\}$ contains $k \ge d+1$ points, then the moduli space $I_d(Q,\gamma)$ is cut out from this torus by imposing the conditions that $p(q_i) \in \gamma$ for each value $d+1 < i \le k$.
Heuristically, each point condition cuts down the dimension by 1, the codimension of $\gamma \subset \bR^2$.
Throwing out the constant inscriptions, it follows that $I_d^*(Q,\gamma)$ has expected dimension $d+1-(k-d-1) = 2(d+1)-k$.
In particular, we expect $I_d^*(Q,\gamma)$ to consist of a finite number of points when $|Q|=2(d+1)$.
This threshold case is a natural case to consider and hence the focus of this paper.

\subsection*{Plan of the paper.}
\Cref{sec:dud2} develops the symplectic method necessary to prove \Cref{thm:quadratic_inscriptions}, and \Cref{sec:pinwheel} develops the Floer theory necessary to prove \Cref{prop:pinwheel}.

\subsection*{Convention.}
We set $n = d+1$.

\subsection*{Acknowledgements.}  The authors gratefully acknowledge conversations or exchanges of emails with
Paul Biran,
Octav Cornea,
Kenji Fukaya,
and
Yank{\i} Lekili.
\newpage
\section{Polynomials, circles, and symplectic forms.}
\label{sec:dud2}
In this section, we consider the problem of finding a $\bC[z]^*_{n-1}$-inscription of a set of $2n$ points $Q \subset S^1$ into a smooth Jordan curve $\gamma \subset \bC$, recasting it into a problem in symplectic geometry.
\label{subsec:summary_of_symplectic_situation}
We start with $2n$ distinct cyclically ordered points
$\alpha_1, \beta_1, \ldots, \alpha_n, \beta_n \in \bC$.
From this we will define a map $F^\beta_\alpha \in {\rm GL}_n(\bC)$
so that intersection points $\gamma^n \cap F^\beta_\alpha(\gamma^n) \subset \bC^n$ correspond to inscriptions $p \in \bC[z]_{n-1}$ of $\{  \alpha_1, \beta_1, \ldots, \alpha_n, \beta_n \}$ into $\gamma$ (\Cref{prop:intersections_correspond_to_inscriptions}).
Furthermore, when the points are concyclic, we shall exhibit a symplectic $2$-form $\psi^\alpha_\beta \in \Omega^2(\bC^n)$ (\Cref{prop:simultaneously_lagrangian}) with the property that $\gamma^n, F^\beta_\alpha(\gamma^n) \subset \bC^n$ are both Lagrangian with respect to $\psi^\alpha_\beta$ (\Cref{cor:simultaneously_lagrangian}).

We note that $(\bC^n, \psi^\alpha_\beta)$ is symplectomorphic via a diagonal linear map to $(\bC^n,\omegast)$, so that we may appeal to results about this latter space.

Define the diagonal loop
\[
\Delta(\gamma) := \{ (a,a, \ldots, a) : a \in \gamma \} \subset \bC^n.
\]
It is a clean loop of intersection between the two Lagrangians $\gamma^n$, $F^\beta_\alpha(\gamma^n)$ (\Cref{cor:clean}).
The points of $\Delta(\gamma)$ correspond to the constant inscriptions (\Cref{cor:off-diagonal_intersections_are_kewl}), so that self-intersections of the immersed Lagrangian smoothing $L$ of $\gamma^n \cup F^\beta_\alpha(\gamma^n)$ along $\Delta(\gamma)$ correspond to non-constant polynomial inscriptions $p \in \bC[z]_{n-1}^*$.

More precisely, $L \subset \bC^n$ is the image of a Lagrangian immersion of $S^1 \times (\mathbb{T}^{n-1} \# \mathbb{T}^{n-1})$ in $\bC^n$ (\Cref{lem:lag_surg_is_immersion}).  The circle fiber $S^1 \times \{ {\rm pt} \}$ has Maslov index $2n > 2$.
Hence we can conclude \Cref{prop:conj_implies_conj}, that \Cref{conj:maslov2} implies \Cref{conj:cyc_poly_peg_prob}.

\subsection{Two tori in $\bC^n$.}
We now develop the set-up in earnest.

\label{subsec:recasting_as_intersection_problem}
Let $\alpha = (\alpha_1,\dots,\alpha_n)$ denote an $n$-tuple of distinct points of $\bC$.
The map
\[
\mathrm{ev}_\alpha \colon \bC[z]_{n-1} \longrightarrow \bC^n \colon p \longmapsto (p(\alpha_1),\dots,p(\alpha_n))
\]
defines a linear isomorphism between complex vector spaces.
With respect to the basis $\{1,z,z^2,\dots,,z^{n-1}\}$ of $\bC[z]_{n-1}$ and the standard basis of $\bC^n$, it is given by the Vandermonde matrix
\[
V^\alpha =
\left(
\begin{matrix}
	1 & \alpha_1 & \alpha_1^2 & \dots & \alpha_1^{n-2} & \alpha_1^{n-1} \\
	1 & \alpha_2 & \alpha_2^2 & \dots & \alpha_2^{n-2} & \alpha_2^{n-1} \\
	\vdots & \vdots & \vdots & \ddots & \vdots & \vdots \\
	1 & \alpha_n & \alpha_n^2 & \dots & \alpha_n^{n-2} & \alpha_n^{n-1} \\
\end{matrix}
\right)
.
\]
If $\beta=(\beta_1,\dots,\beta_n)$ denotes another $n$-tuple of distinct points of $\bC$, then we define the $\bC$-linear automorphism
\[
F_\alpha^\beta := \mathrm{ev}_\beta \circ (\mathrm{ev}_\alpha)^{-1} : \bC^n \longrightarrow \bC^n {\rm.}
\]
Write $Q = \{\alpha_1,\dots,\alpha_n,\beta_1,\dots,\beta_n\}$, and recall the moduli spaces of inscriptions
\[
I_{n-1}^*(Q,\gamma) := \{ p(z) \in \bC[z]_{n-1}^* : p(Q) \subset \gamma \},
\]
\[
I_{n-1}(Q,\gamma) := \{ p(z) \in \bC[z]_{n-1} : p(Q) \subset \gamma \}.
\]

\begin{prop}
	\label{prop:intersections_correspond_to_inscriptions}
	\label{cor:off-diagonal_intersections_are_kewl}
	Given a Jordan curve $\gamma \subset \bC$, there is a one-to-one correspondence
	\[
	\gamma^n \cap F^\beta_\alpha(\gamma^n) \overset{\sim}{\longrightarrow} I_{n-1}(Q,\gamma),
	\]
	and it carries $(\gamma^n \cap F^\beta_\alpha(\gamma^n)) \setminus \Delta(\gamma)$ to $I_{n-1}^*(Q,\gamma)$.
\end{prop}

\begin{proof}
	Elements of $\gamma^n \cap F^\beta_\alpha(\gamma^n)$ correspond via $\ev_\beta^{-1}$
	to elements of $\ev_\beta^{-1}(\gamma^n) \cap \ev_\alpha^{-1} (\gamma^n) \subset \C[z]_{n-1}$.
	Such elements are polynomials $p \in \bC[z]_{n-1}$ satisfying
	\[ (p(\alpha_1), \ldots, p(\alpha_n)), (p(\beta_1), \ldots , p(\beta_n)) \in \gamma^n, \]
	which is to say that $p(z) \in I_{n-1}(Q,\gamma)$.
	Moreover, $\ev_\beta^{-1}$ carries the thin diagonal $\Delta(\bC) \subset \bC^n$ into the set of constant polynomials $\bC \subset \bC[z]_{n-1}$.
	Hence it carries $\Delta(\gamma)$ into the set of constant inscriptions.
	This establishes the second part of the result.
\end{proof}

\subsection{Two Lagrangians in $\bC^n$.}
In this subsection we shall restrict to the case when the
point set is contained in the unit circle: $Q \subset S^1$.
We shall then use this assumption to find a symplectic form on $\bC^n$ for which $\gamma^n$ and $F_\alpha^\beta(\gamma^n)$ are simultaneously Lagrangian.

\label{subsec:checking_lagrangian}
Let $z_k = x_k + \sqrt{-1} y_k$, $k=1,\dots,n$, denote coordinates on $\bC^n = \bR^{2n}$.
\begin{defin}
A complexified 2-form $\omega \in \Omega^2_{\bC}\bC^n = \Omega^2(\bC^n) \otimes \bC$ is called \emph{diagonal} if it takes the form
\[
\omega = \sum_{k=1}^n \lambda_k \cdot d z_k \wedge d \overline{z_k} 
\]
for some coefficients $\lambda_k \in \bC$.  The form $\omega$ is called \emph{real} if each $\lambda_k \in \bR \sqrt{-1}$ and \emph{positive} if furthermore $\lambda_k \sqrt{-1} < 0$.
\end{defin}

The reason for the notation is firstly that positive real diagonal $2$-forms $\omega$ live in the subspace $\Omega^2(\bR^{2n}) \subset \Omega^2_{\bC}\bC^n$ (hence `real').  Secondly, when written in the standard basis of $\Omega^2(\bR^{2n})$, they take the form
\[ \omega = \sum_{k=1}^n \mu_k \cdot d x_k \wedge d y_k \]
where each $\mu_k > 0$ (hence `positive diagonal').  Note that $\gamma^n$ is Lagrangian for any such form -- we wish to find such a form for which $F_\alpha^\beta(\gamma^n)$ is also Lagrangian.   Working in the bigger complexified space $\Omega^2_{\bC}\bC^n$ will be convenient for our proofs.

\begin{prop}
	\label{prop:simultaneously_lagrangian}
	Suppose that $\alpha, \beta \in (S^1)^{n}$ are $n$-tuples of points on the unit circle such that
	\[ (\alpha_1, \beta_1, \alpha_2, \beta_2, \ldots, \alpha_n, \beta_n) \in (S^1)^{2n} \]
	is a cyclically ordered $2n$-tuple of distinct points.  Then we have the following:

	\begin{enumerate}[(i)]	
	\item If we write $\Delta_n \subset \Omega_\bC^2(\bC^n)$ for the subspace of diagonal forms, then the complex vector space
	\[ \Delta_n \cap (F^\beta_\alpha)^* \Delta_n \] is $1$-dimensional.
	
	\item Furthermore, we can find positive real diagonal forms $\psi_\alpha^\beta$ and $\psi^\alpha_\beta$ such that 
	\[ (F^\beta_\alpha)^* \psi^\alpha_\beta = \psi_\alpha^\beta {\rm .} \]
	\end{enumerate}
\end{prop}

The point of this proposition is the following corollary.
\begin{cor}
	\label{cor:simultaneously_lagrangian}
	Both $\gamma^n$ and $F_\alpha^\beta(\gamma^n)$ are Lagrangian with respect to the symplectic form $\psi_\beta^\alpha$.
\end{cor}
\begin{proof}
	Note that $\gamma^n$ is Lagrangian with respect to $\psi^\beta_\alpha = (F^\beta_\alpha)^*(\psi^\alpha_\beta)$, so that $F^\beta_\alpha(\gamma^n)$ is Lagrangian with respect to $\psi^\alpha_\beta$.  And $\gamma^n$ is also itself Lagrangian with respect to this form $\psi^\alpha_\beta$, so we are done.
\end{proof}

The second part of the statement of Proposition~\ref{prop:simultaneously_lagrangian} is therefore what mainly interests us.  The first part is there to guide us through the proof.

\begin{proof}[Proof of \Cref{prop:simultaneously_lagrangian}.]
		In what follows we identify $\bC[z]_{n-1}$ with $\C^n$ via the linear isomorphism
	\[ a_0 + a_1 z + \ldots a_{n-1} z^{n-1} \longmapsto \left( \begin{matrix} a_0 \\ a_1 \\ \vdots \\ a_{n-1} \end{matrix} \right) \]
	
(i)	We begin by looking for diagonal forms $\omega^\beta_\alpha, \omega^\alpha_\beta \in \Delta_n$ such that
	\[ \omega_\alpha^\beta = (F^\beta_\alpha)^* \omega_\beta^\alpha = ( (\ev_\alpha^*)^{-1} \circ \ev_\beta^*  )\omega_\beta^\alpha \]
	or, equivalently,
	\[ (\ev_\alpha^*) \omega_\alpha^\beta = (\ev_\beta^*)  \omega_\beta^\alpha \in \Omega_\bC^2(\bC[z]_{n-1}){\rm .} \]
	
	If we write
	\[ \omega_\alpha^\beta = \sum_{k=1}^n \lambda_k \cdot d z_k \wedge d \overline{z_k} \,\,\,\, {\rm and} \,\,\,\, \omega_\beta^\alpha = \sum_{k=1}^n \mu_k \cdot d z_k \wedge d \overline{z_k}\]
	then we have
	\begin{align*}
		(\ev_\alpha^*) \omega_\alpha^\beta &= \sum_{k=1}^n \lambda_k \cdot (\ev_\alpha^* d z_k) \wedge (\ev_\alpha^* d \overline{z_k}) = \sum_{1 \leq i,j,k \leq n} \lambda_k \cdot (V^\alpha_{ki} d z_i) \wedge (\ol{V^\alpha_{kj}}d \overline{z_j}) \\
		&= \sum_{1 \leq i,j,k \leq n} \lambda_k \cdot V^\alpha_{ki} \ol{V^\alpha_{kj}} \cdot d z_i \wedge d \overline{z_j} = \sum_{1 \leq i,j,k \leq n} \lambda_k \cdot \alpha_{k}^{i-1} \ol{\alpha_{k}^{j-1}} \cdot d z_i \wedge d \overline{z_j} \\
		 &=
		\sum_{1 \leq i,j,k \leq n} \lambda_k \cdot \alpha_{k}^{i-j} \cdot d z_i \wedge d \overline{z_j}
	\end{align*}
	and similarly
	\[ (\ev_\beta^*) \omega^\alpha_\beta = \sum_{1 \leq i,j,k \leq n} \mu_k \cdot \beta_{k}^{i-j} \cdot d z_i \wedge d \overline{z_j} {\rm .} \]
	
	It follows that $(\ev_\alpha^*) \omega_\alpha^\beta = (\ev_\beta^*)  \omega_\beta^\alpha$ exactly when we have
	\begin{equation}
		\sum_{k=1}^n \lambda_k \alpha_k^{i-j} - \mu_k \beta_k^{i-j} = 0
	\end{equation}
	for all choices $1 \leq i, j \leq n$.  In other words, we are looking for coefficients $\lambda_k, \mu_k \in \bC$ for $1 \leq k \leq n$ that satisfy the matrix equation
	\[ \left( \begin{matrix} \lambda_1 & \lambda_2 & \cdots & \lambda_n & -\mu_1 & \cdots & -\mu_n \end{matrix} \right) \left( \begin{matrix} \alpha_1^{1-n} & \alpha_1^{2-n} & \cdots & 1 & \cdots & \alpha_1^{n-1} \\
	\alpha_2^{1-n} & \alpha_2^{2-n} &\cdots & 1 & \cdots & \alpha_2^{n-1} \\
	\vdots & \vdots &\ddots & \vdots & \ddots & \vdots \\
	\alpha_n^{1-n} & \alpha_n^{2-n} & \cdots & 1 &\cdots & \alpha_n^{n-1} \\
	\beta_1^{1-n} & \beta_1^{2-n} &\cdots & 1 & \cdots & \beta_1^{n-1} \\
	\vdots & \vdots &\ddots & \vdots & \ddots & \vdots \\
	\beta_n^{1-n} & \beta_n^{2-n} &\cdots & 1 & \cdots & \beta_n^{n-1} \\
	 \end{matrix} \right) = 0 {\rm .} \]
	 The matrix appearing in this equation with $2n$ rows and $2n-1$ columns we call $V$.  The square matrix obtained by deleting the final row we call $W$.
	 
	 The determinant of $W$ is non-zero, since it becomes a Vandermonde matrix after multiplying each row by the final entry of that row.  Hence $V$ is of rank $2n-1$ so there is a $1$-dimensional space of vectors $\left( \begin{matrix} \lambda_1 & \lambda_2 & \cdots & \lambda_n & -\mu_1 & \cdots & -\mu_n \end{matrix} \right)$ satisfying the equation, from which we fix a non-zero vector to give the coefficients $\lambda_k, \mu_k \in \bC$ of $\omega^\alpha_\beta$ and $\omega_\alpha^\beta$.

(ii)	Now we wish to see that there exists a complex number $\sigma \in \bC^\times$ such that
	 \[ \psi^\alpha_\beta = \sigma \omega^\alpha_\beta \,\,\,\, {\rm and} \,\,\,\, \psi_\alpha^\beta = \sigma \omega_\alpha^\beta \]
	 are both positive real diagonal forms.  We shall take $\sigma = \sqrt{-1}/\mu_n$.  So it remains to show that $\mu_n \not= 0$ and that $\lambda_j / \mu_n$ and $\mu_j / \mu_n$ are positive real numbers for $1 \leq j \leq n$.
	 
	 We begin by postmultiplying the matrix equation in (i) by $W^{-1}$.  This results in the matrix equation
	 \[ \left( \begin{matrix} \lambda_1 & \lambda_2 & \cdots & \lambda_n & -\mu_1 & \cdots & -\mu_n \end{matrix} \right) \left( \begin{matrix} 1 & 0 & \cdots & 0 \\
	 	0 & 1 &\cdots & 0 \\
	 	\vdots & \vdots &\ddots & \vdots \\
	 	0 & 0 & \cdots & 1 \\
	 	(\beta_n^{1-n} & \beta_n^{2-n} &\cdots & \beta_n^{n-1}) W^{-1}
	 \end{matrix} \right) = 0 {\rm .} \]
	 From this we see firstly that $\mu_n \not= 0$, since otherwise only the zero vector would be a solution to this equation.  Secondly, from this equation we are able to read off the quotients $\lambda_j / \mu_n$ and $\mu_j / \mu_n$; indeed, if we write $\rho_j$ for the $j$th entry of the final row of the second matrix in this product then we have
	 \[ \frac{\lambda_j}{\mu_n} = \rho_j \,\,\,\, {\rm for} \,\,\,\, 1 \leq j \leq n {\rm ,}\,\,\,\, {\rm and} \,\,\,\, \frac{\mu_{j}}{\mu_n} = -\rho_{j+n} \,\,\,\, {\rm for} \,\,\,\, 1 \leq j \leq n-1 {\rm .}\]
	 
	 We compute for $1 \leq j \leq n-1$ (where the second equality follows from properties of Vandermonde inverses)
	 \begin{align*}
	 	\frac{\lambda_j}{\mu_n} &= \rho_j =  \left(\frac{\alpha_j}{\beta_n}\right)^{n-1} \prod_{k \not= j} \frac{\beta_n - \alpha_k}{\alpha_j - \alpha_k} \prod_{k \not= n}\frac{\beta_n - \beta_k}{\alpha_j - \beta_k} \\
	 	&= \prod_{k \not= j} \left(\frac{\alpha_j}{\beta_n}\right)\frac{\beta_n - \alpha_k}{\alpha_j - \alpha_k} \prod_{k \not= n}\frac{\beta_n - \beta_k}{\alpha_j - \beta_k}= \prod_{k \not= j} \frac{1 - \alpha_k \ol{\beta_n}}{1 - \alpha_k\ol{\alpha_j}} \prod_{k \not= n}\frac{\beta_n - \beta_k}{\alpha_j - \beta_k} \\
	 	&= \prod_{k \not= j} \frac{\ol{\beta_n} - \ol{\alpha_k}}{\ol{\alpha_j} - \ol{\alpha_k}} \prod_{k \not= n}\frac{\beta_n - \beta_k}{\alpha_j - \beta_k} =
	 	\prod_{k \not= j} \left\vert \frac{\beta_n - \alpha_k}{\alpha_j - \alpha_k} \right\vert^2 \frac{\alpha_j - \alpha_k}{\beta_n - \alpha_k}
	 	 \prod_{k \not= n}\frac{\beta_n - \beta_k}{\alpha_j - \beta_k} \\
	 	 &= \left\vert \frac{\beta_n - \alpha_n}{\alpha_j - \alpha_n} \right\vert^2 \frac{(\alpha_j - \alpha_n)(\beta_n - \beta_j)}{(\alpha_j - \beta_j)(\beta_n - \alpha_n)}
	 	 \prod_{k \not= j,n} \left\vert \frac{\beta_n - \alpha_k}{\alpha_j - \alpha_k} \right\vert^2 \frac{(\alpha_j - \alpha_k)(\beta_n - \beta_k)}{(\alpha_j - \beta_k)(\beta_n - \alpha_k)} \\
	 	 &=\left\vert \frac{\beta_n - \alpha_n}{\alpha_j - \alpha_n} \right\vert^2 [\alpha_j, \beta_n; \alpha_n , \beta_j] \prod_{k \not= j,n} \left\vert \frac{\beta_n - \alpha_k}{\alpha_j - \alpha_k} \right\vert^2 [\alpha_j, \beta_n; \alpha_k, \beta_k] > 0
	 \end{align*}
	 where the final inequality follows since every cross-ratio appearing is positive.  For $j=n$ we have similarly
	 \begin{align*}
	 	\frac{\lambda_n}{\mu_n} &= \rho_n =  \left(\frac{\alpha_n}{\beta_n}\right)^{n-1} \prod_{k \not= n} \frac{\beta_n - \alpha_k}{\alpha_n - \alpha_k} \prod_{k \not= n}\frac{\beta_n - \beta_k}{\alpha_n - \beta_k} \\
	 	&= \prod_{k \not= n} \left\vert \frac{\beta_n - \alpha_k}{\alpha_n - \alpha_k} \right\vert^2 \frac{(\alpha_n - \alpha_k)(\beta_n - \beta_k)}{(\alpha_n - \beta_k)(\beta_n - \alpha_k)} = \prod_{k \not= n} \left\vert \frac{\beta_n - \alpha_k}{\alpha_n - \alpha_k} \right\vert^2 [\alpha_n, \beta_n; \alpha_k, \beta_k] > 0 {\rm .}
	 \end{align*}
	 
	 And on the other hand we have
	 \begin{align*}
	 	\frac{\mu_{j}}{\mu_n} &= -\rho_{j+n} = - \left( \frac{\beta_j}{\beta_n} \right)^{n-1} \prod_k \frac{\beta_n - \alpha_k}{\beta_j - \alpha_k} \prod_{k \not= j,n} \frac{\beta_n - \beta_k}{\beta_j - \beta_k} \\
	 	&= -\frac{\beta_n - \alpha_n}{\beta_j - \alpha_n}\prod_{k \not= n} \left( \frac{\beta_j}{\beta_n} \right) \frac{\beta_n - \alpha_k}{\beta_j - \alpha_k} \prod_{k \not= j,n} \frac{\beta_n - \beta_k}{\beta_j - \beta_k} \\
	 	&=-\frac{\beta_n - \alpha_n}{\beta_j - \alpha_n} \prod_{k \not= n} 
	 	 \left\vert \frac{\beta_n - \alpha_k}{\beta_j - \alpha_k} \right\vert^2 \frac{\beta_j - \alpha_k}{\beta_n - \alpha_k}\prod_{k \not= j,n} \frac{\beta_n - \beta_k}{\beta_j - \beta_k} \\
	 	&= -\left\vert \frac{\beta_n - \alpha_j}{\beta_j - \alpha_j} \right\vert^2 \frac{(\beta_n - \alpha_n)(\beta_j - \alpha_j)}{(\beta_n - \alpha_j)(\beta_j - \alpha_n)} \prod_{k \not= j,n} \left\vert \frac{\beta_n - \alpha_k}{\beta_j - \alpha_k} \right\vert^2 \frac{(\beta_j - \alpha_k)(\beta_n - \beta_k)}{(\beta_j - \beta_k)(\beta_n - \alpha_k)} \\
	 	&= -\left\vert \frac{\beta_n - \alpha_j}{\beta_j - \alpha_j} \right\vert^2 [\beta_n,\beta_j;\alpha_n,\alpha_j] \prod_{k \not= j,n} \left\vert \frac{\beta_n - \alpha_k}{\beta_j - \alpha_k} \right\vert^2 [\beta_j,\beta_n;\alpha_k,\beta_k] >  0
	 \end{align*}
	 where the final inequality follows since all cross-ratios appearing in the product are positive apart from $[\beta_n,\beta_j;\alpha_n,\alpha_j]$, which is negative.
\end{proof}

\subsection{Polynomials, circles, and real subspaces.}
\label{subsec:clean_intersection}
The previous subsection showed that the problem of finding intersection points between the two tori $\gamma^n$ and $F^\beta_\alpha(\gamma^n)$ is a symplectic problem of finding intersection points between two Lagrangians.
As we have noted, there is a circle of intersection points corresponding to inscriptions by constant polynomials.
In this subsection we show that the circle of intersection is clean.
This property is necessary both for performing Lagrangian surgery in the proof of \Cref{prop:conj_implies_conj} and for studying the pearl Floer complex in \Cref{sec:pinwheel}.

\begin{prop}
	\label{prop:clean}
	Suppose that $\alpha_1,\beta_1,\dots,\alpha_n,\beta_n$ are distinct cyclically ordered points on $S^1$.
	Then 
	\begin{itemize}
		\item $F_\alpha^\beta(\bR^n) \cap \bR^n = \langle {\bf 1} \rangle$, where ${\bf 1} =(1,\dots,1)$, and
		\item $F_\alpha^\beta({\bf 1}) = {\bf 1}$, i.e. ${\bf 1}$ is an eigenvector of $F_\alpha^\beta$ with eigenvalue 1.
	\end{itemize}
\end{prop}

The point of this proposition is the following corollary.

\begin{cor}
	\label{cor:clean}
		The diagonal loop $\Delta(\gamma) = \{ a \cdot {\bf 1} : a \in \gamma \}$ is a clean component of intersection between $\gamma^n$ and $F^\beta_\alpha(\gamma^n)$.
\end{cor}

\begin{proof}
	Let $a \in \gamma$ and let $v \in T_a \gamma \subset T_a \bC \cong \bC$ denote a nonzero tangent vector.
	Then
	\[ T_{(a \cdot {\bf 1})}\gamma^n = \left\{ (\lambda_1 v, \ldots, \lambda_n v) : \lambda_i \in \bR, 1 \leq  i \leq n \right\} = v \bR^n \subset \bC^n \cong \oplus_{k=1}^n T_a \bC {\rm .} \]
	and
	\begin{align*}
		T_{(a \cdot {\bf 1})}F^\beta_\alpha(\gamma^n) &= (F^\beta_\alpha)_* T_{(a \cdot {\bf 1})}(\gamma^n) = (F^\beta_\alpha)_* (v \bR^n) = v (F^\beta_\alpha)_* (\bR^n)  \subset \bC^n \cong \oplus_{k=1}^n T_a \bC {\rm .}
	\end{align*}
	The differential of a linear map agrees with the linear map after identifying the tangent spaces with the underlying ambient space.
	Since $F^\beta_\alpha(\bR^n) \cap \bR^n = \langle {\bf 1} \rangle$, it follows that
	\[
	T_{(a \cdot {\bf 1})} \gamma^n \cap T_{(a \cdot {\bf 1})} F^\beta_\alpha(\gamma^n) = v(\bR^n \cap F^\beta_\alpha(\bR^n)) = v \langle {\bf 1} \rangle = T_{(a \cdot {\bf 1})} \Delta(\gamma).
	\]
	Hence $\Delta(\gamma)$ is a clean component of intersection of $\gamma^n \cap F^\beta_\alpha(\gamma^n)$.
\end{proof}

\begin{proof}[Proof of \Cref{prop:clean}]
	Suppose that
	\[
	v \in \bR^n \quad \textup{and} \quad F_\alpha^\beta(v) = w \in \bR^n.
	\]
	Then there exists a polynomial $p \in \bC[z]_{n-1}$ such that
	\[
	\mathrm{ev}_\alpha(p) = v \quad \textup{and} \quad \mathrm{ev}_\beta(p) = w.
	\]
	In particular, $p$ maps the $2n$ distinct values $\alpha_1,\beta_1,\dots,\alpha_n,\beta_n \in S^1$ to $\bR$.
	Write $p(x+iy) = p_1(x,y) + i p_2(x,y)$ with $x,y$ real variables, $p_1,p_2 \in \bR[x,y]$, and $\deg p_1, \deg p_2 \le n-1$.
	The points $x+iy$ which get sent by $p$ to $\bR$ satisfy the polynomial equation $p_2(x,y) = 0$.
	The points on $S^1$ satisfy $x^2 + y^2 - 1 = 0$.
	These are two polynomial equations in two indeterminates of respective degrees $\le n-1$ and 2.
	Hence the common solution set has cardinality $\le 2 (n-1)$ or else contains a component of positive dimension of the individual solution sets.
	Here we invoke B\'ezout's theorem, which for plane curves was already known to Newton \cite{bezout,newton}.
	Since $p$ maps $2n$ distinct points on $S^1$ to $\bR$, the second case must occur.
	Since $x^2+y^2-1$ is irreducible, this implies that $p(S^1) \subset \bR$.
	Since $p$ is holomorphic, the maximum principle implies that $p(D^2)$ is contained in the region bounded by $p(S^1) \subset \bR$.
	This implies that $p$ is constant: $p = \lambda \in \bR$.
	Hence $v=w= \lambda \cdot {\bf 1}$, which completes the proof.
\end{proof}

\subsection{Lagrangian smoothing and \Cref{prop:conj_implies_conj}.}
We now wish to consider the result of a Lagrangian smoothing of $\gamma^n \cup F^\beta_\alpha(\gamma^n)$ along the circle of clean intersection $\Delta(\gamma)$.

\begin{lem}
	\label{lem:lag_surg_is_immersion}
	The Lagrangian smoothing of $\gamma^n \cup F^\beta_\alpha(\gamma^n)$ along $\Delta(\gamma)$ is a Lagrangian immersion of $S^1 \times (\mathbb{T}^{n-1} \# \mathbb{T}^{n-1})$ into $\bC^n$.
\end{lem}

\begin{proof}
	The result of the smoothing will be an immersion of a manifold obtained by splicing together two $n$-tori which have had tubular neighborhoods of circle fibers removed.  Precisely, an immersion of
	\[ {\T}^n \setminus ({\T}^1 \times D^{n-1}) \cup_\phi {\T}^n \setminus (\T^1 \times D^{n-1}) \]
	where $\T^1$ is a circle factor of $\T^n$ and $\phi$ is a gluing map given by a smooth choice of identifications ${\rm SO}_{n-1}(\bR) \ni \phi_p  \colon \{ p \} \times S^{n-2} \rightarrow \{ p \} \times S^{n-2}$ for $p \in \T^1$.  So there are \emph{a priori} infinitely many possibilities for the result of the smoothing for $n=3$, and two possibilities for each $n \geq 4$.  We wish to verify that $\phi_p$ is a null-homotopic loop in ${\rm SO}_{n-1}(\bR)$ to conclude that the smoothing is an immersion of $S^1 \times (\mathbb{T}^{n-1} \# \mathbb{T}^{n-1})$.
	
	Pozniak's thesis provides a local picture for the neighbourhood of a clean intersection \cite{pozniak}.  In our case, when the clean intersection is topologically a circle, the local picture is the manifold
	\[ M = S^1 \times (-1,1) \times (-1,1)^{n-1} \times (-1,1)^{n-1} \]
	with coordinates $s\in \bR/\bZ,t, x_1, \ldots x_{n-1}, y_1, \ldots, y_{n-1}$ and symplectic form
	\[ \omega = ds \wedge dt + dx_1 \wedge dy_1 + \cdots dx_{n-1} \wedge dy_{n-1} {\rm ,} \]
	which contains two Lagrangians
	\[ L_X = \{ p \in M : t(p) = y_1(p) = \cdots = y_{n-1}(p) = 0 \}, L_Y = \{ p \in M : t(p) = x_1(p) = \cdots = x_{n-1}(p) = 0 \} \]
	cleanly intersecting along the central circle of $M$.
	
	We note that for $p \in M$ the images of $\{ \frac{\partial}{\partial x_i} : 1 \leq i \leq n-1 \}$ and of $\{ \frac{\partial}{\partial y_i} : 1 \leq i \leq n-1 \}$ are dual bases of transverse Lagrangian planes inside $T_p M / \langle \frac{\partial}{\partial s} , \frac{\partial}{\partial t} \rangle$ with respect to the induced symplectic form on the quotient.
	
	For each $\lambda \in S^1$, we see two transversely intersecting Lagrangians with respect to this quotient form
	\[ L_{X,\lambda}, L_{Y,\lambda} \subset M_\lambda := \{ p \in M : s(p) = \lambda, t(p) = 0 \} \]
	where $L_{X,\lambda} = L_X \cap M_\lambda$ and $L_{Y,\lambda} = L_Y \cap M_\lambda$.  The smoothing of the clean intersection is performed by making the same smoothing of each of these transverse intersections as $\lambda$ varies over $S^1$ (how one smooths a transverse intersection is described by Polterovich \cite{polterovich_smoothing}).  This discussion indicates how to construct our reference loop in $\mathrm{SO}_{n-1}(\bR)$, which we now proceed to do.
	
	For each $a \in \gamma$, choose an ordered basis $\cB_a$ of $T_{a \cdot {\bf 1}} \gamma^n / T_{a \cdot {\bf 1}} \Delta(\gamma)$.  With respect to the symplectic form induced by $\psi^\alpha_\beta$ on the quotient $T_{a \cdot {\bf 1}}\bC^n/T_{a \cdot {\bf 1}} \Delta(\bC)$, we let $\cC_a$ be the ordered basis of $T_{a \cdot {\bf 1}}F^\beta_\alpha(\gamma^n) / T_{a \cdot {\bf 1}} \Delta(\gamma)$ that is dual to $\cB_a$.
	Then $(F^\beta_\alpha)^{-1}_* \cC_a = (F^\beta_\alpha)^{-1} \cC_a$ gives another basis for $T_{a \cdot {\bf 1}} \gamma^n / T_{a \cdot {\bf 1}} \Delta(\gamma)$ (here we have abused notation by writing $(F^\beta_\alpha)^{-1}$ for the linear map induced by $(F^\beta_\alpha)^{-1}$ on a subquotient).  By comparing these bases we get a loop of elements of ${\rm GL}_{n-1}(\bR)$ parametrized by $a \in \gamma$.  By noting that orientations must be preserved we thus get a loop in ${\rm SO}_{n-1}(\bR)$ following a deformation retraction.
	
	We next proceed to make good choices of $\cB_a$ to show that this resulting loop is contractible.
	
	We first choose a unit-speed parametrisation of $\gamma$, making the abuse of notation between $\gamma$ and its parametrization $\gamma \colon \bR \rightarrow \bC$, and arrange that $\gamma'(0) = 1 \in \bC = T_{\gamma(0)} \bC$.  Then we choose a basis $\cB_0$ for $T_{\gamma(0) \cdot {\bf 1}} \gamma^n / T_{\gamma(0) \cdot {\bf 1}} \Delta(\gamma)$.  Noting that multiplication by non-zero complex numbers gives real-linear automorphisms of $\bC^n$, for each $s \in \bR$ we take the basis $\cB_s = \gamma'(s)\cB_0$ of $T_{\gamma(s) \cdot {\bf 1}} \gamma^n / T_{\gamma(s) \cdot {\bf 1}} \Delta(\gamma)$.  As above, we write $\cC_s$ for the corresponding dual basis of $T_{\gamma(s) \cdot {\bf 1}}F^\beta_\alpha(\gamma^n) / T_{\gamma(s) \cdot {\bf 1}} \Delta(\gamma)$.  Note that $\gamma'(s) \cC_0$ is also a basis for $T_{\gamma(s) \cdot {\bf 1}}F^\beta_\alpha(\gamma^n) / T_{\gamma(s) \cdot {\bf 1}} \Delta(\gamma)$ by complex-linearity of $F^\beta_\alpha$.  Indeed, since $\psi^\alpha_\beta$ is real diagonal, we see that $\gamma'(s) \cC_0$ is dual to $\gamma'(s) \cB_0 = \cB_s$, so that we have $\cC_s = \gamma'(s) \cC_0$.

	Finally we obtain the loop in ${\rm GL}_{n-1}(\bR)$ by comparing the bases $\cB_s = \gamma'(s) \cB_0$ with the bases
	\[ (F^\beta_\alpha)^{-1}(\cC_s) = (F^\beta_\alpha)^{-1}(\gamma'(s) \cC_0) = \gamma'(s)(F^\beta_\alpha)^{-1}(\cC_0) {\rm .} \]
	But this just gives a constant loop, and in particular a contractible loop.
\end{proof}

We are now ready to conclude \Cref{prop:conj_implies_conj}.

\begin{proof}[Proof of \Cref{prop:conj_implies_conj}.]
	We have seen that points of $(\gamma^n \cap F^\beta_\alpha(\gamma^n)) \setminus \Delta(\gamma)$ correspond to inscriptions $p \in \bC[z]_{n-1}^*$ of $\{\alpha_1, \ldots, \alpha_n, \beta_1, \ldots, \beta_n \}$ in $\gamma$ (\Cref{cor:off-diagonal_intersections_are_kewl}).  \Cref{lem:lag_surg_is_immersion} establishes that these points correspond to self-intersections of a Lagrangian immersion of $S^1 \times (\mathbb{T}^{n-1} \# \mathbb{T}^{n-1})$ into $\bC^n$.
	
	In the homology of the torus $\gamma^n$, the loop $\Delta(\gamma)$, when oriented appropriately, represents the element
	\[ (1,1, \ldots, 1) \in  \bigoplus_{i=1}^n H_1(\gamma, \bZ) = H_1(\gamma^n, \bZ) \]
	via the K{\"u}nneth identification.  Since the Maslov index of a loop depends only on its homology class, we conclude that the Maslov index of $\Delta(\gamma)$ is $2n$ with respect to $\gamma^n$.
	
	Thus the circle factor of $S^1 \times (\mathbb{T}^{n-1} \# \mathbb{T}^{n-1})$ is of Maslov index $2n$.  Since \Cref{conj:maslov2} tells us that this loop is of Maslov index $2$ when the immersion is an embedding, it must be that the immersion has points of self-intersection.
\end{proof}

\newpage
\section{Pinwheels.}
\label{sec:pinwheel}
This section establishes \Cref{prop:pinwheel}, which concerns polynomial inscriptions of pinwheels.
\Cref{sec:pinwheeltori} establishes nice features for the pair of parametrizing tori for a pinwheel.
In particular, their Floer homology is unobstructed.
We then explain how \Cref{prop:pinwheel} follows from a computation in pearl Floer homology (\Cref{prop:pinwheel_homology_does_not_vanish}).
\Cref{sec:pearlreview} reviews the construction of the pearl Floer homology of a pair of cleanly intersecting Lagrangians.
\Cref{sec:pearlpinwheel} then specializes the construction to the case of the pinwheel tori.
It concludes with the proof of \Cref{prop:pinwheel_homology_does_not_vanish}.

\subsection{Pinwheel tori.}
\label{sec:pinwheeltori}
We specialize the construction of \Cref{sec:dud2} to the case of pinwheels.
Fix $n \ge 2$.

\begin{defin}
	\label{defin:pinwheel_vertices}
	Let $\omega = e^{2i\pi/n}$ be a primitive $n$th root of unity.  For $1 \leq j \leq n$ and $0 < \theta < 2\pi/n$ we define
	\[ \alpha_j = \omega^j \,\,\, {\rm and} \,\,\, \beta_j = e^{\theta\sqrt{-1}}\omega^j \]
	to be the vertices of the $\theta$\emph{-pinwheel} $Q_\theta$.
\end{defin}

\noindent
Thus, $Q_\theta$ is a rectangle whose diagonals meet at angle $\theta$ in the case $n=2$.
Since $F^\beta_\alpha$ only depends on $\theta$, we write $F^\beta_\alpha = F_\theta$.
The cyclic group $\bZ/n$ acts on $\bC^n$ by cyclic permutation of the coordinates, and the action preserves $\omegast$.

Pinwheel tori have the following nice features:

\begin{prop}
	\label{lem:pinwheels_use_standard_form}
	The tori $\gamma^n$ and $F_\theta(\gamma^n)$ are Lagrangian, monotone, $\bZ/n$-equivariant, and Hamiltonian isotopic in $(\bC^n,\omegast)$.
\end{prop}

We build up to its proof after two lemmas.

\begin{lem}
	\label{lem:circleaction}
	The isotopy
	\[
	\bR / (2 \pi \bZ) \longrightarrow \mathrm{GL}_n(\bC) : \theta \longmapsto F_\theta
	\]
	defines a Hamiltonian circle action of $(\bC^n,\omegast)$.
\end{lem}	

\begin{proof}
By definition, the map $F_\theta$ factors as $V^\beta (V^\alpha)^{-1}$.
By inspection, $V^\beta = V^\alpha D_\theta$, where $D_\theta$ denotes the diagonal matrix whose $k$-th diagonal entry equals $e^{(k-1) \theta \sqrt{-1}}$.
The map $D_\theta$ is the time-$\theta$ flow of the Hamiltonian
\[
h \colon (\bC^n,\omegast) \longrightarrow \bR \colon (z_1,\dots,z_n) \longmapsto -\frac12 \sum_{k=1}^n (k-1) |z_k|^2.
\]
Hence $\theta \longmapsto \phi^\theta_h = D_\theta$ defines a Hamiltonian circle action of $\bC^n$.
The map $\frac1{\sqrt{n}} V^\alpha$ is a unitary transformation of $\bC^n$ and hence a symplectomorphism of $(\bC^n,\omegast)$.
(Incidentally, it is the discrete Fourier transform.)
Since it conjugates $D_\theta$ into $F_\theta$, it follows that
$F_\theta$ is the time-$\theta$ flow of the Hamiltonian $H := h \circ (\frac1{\sqrt{n}} V^\alpha)^{-1} : (\bC^n,\omegast) \to \bR$.
Hence $\theta \longmapsto \phi^\theta_H = F_\theta$ defines a Hamiltonian circle action.
\end{proof}

\begin{lem}
	\label{lem:shift}
	The map $F_{2\pi/n}$ acts by a single cyclic shift of the coordinates of $\bC^n$:
	\[
	F_{2\pi/n}(z_1,z_2,\dots,z_n) = (z_2,\dots,z_n,z_1), \quad \forall \, (z_1,\dots,z_n) \in \bC^n.
	\]
\end{lem}

\begin{proof}
When $\theta = 2 \pi/n$, the vector $\beta$ is a single cyclic shift of the coordinates of $\alpha$.
It follows that the evaluation map $\ev_\beta(p)$ is a single cyclic shift of the coordinates of $\ev_\alpha(p)$, for every $p \in \bC[z]_{n-1}$.
The map $F_{2\pi/n}$ sends the second vector to the first, hence it acts as claimed.
\end{proof}

\begin{proof}
[Proof of \Cref{lem:pinwheels_use_standard_form}]
The torus $\gamma^n$ is clearly Lagrangian, monotone, and $\bZ/n$-equivariant.
By \Cref{lem:circleaction}, $F_\theta(\gamma^n)$ is Hamiltonian isotopic to it.
By \Cref{lem:circleaction} and \Cref{lem:shift}, it is $\bZ/n$-equivariant as well, since
\[
F_\theta(\gamma^n) = F_\theta(F_{2\pi/n}(\gamma^n)) = \phi_H^{\theta+2\pi/n}(\gamma^n) = F_{2\pi/n}(F_\theta(\gamma^n)). \qedhere
\]
\end{proof}

We reduce \Cref{prop:pinwheel} to the following Floer theoretic result.
Pearl homology is a variation of Lagrangian Floer homology that we will review in \Cref{sec:pearlreview}.

\begin{prop}
	\label{prop:pinwheel_homology_does_not_vanish}
	Let $n$ be prime.
	Assume that $\gamma^n$ and $F_\theta(\gamma^n)$ intersect only along the locus $\Delta(\gamma)$.
	Then the pearl homology of this pair is defined, unobstructed, and nonzero with $\bZ/n$-coefficients.
\end{prop}

\begin{proof}[Proof of \Cref{prop:pinwheel}]
	By \Cref{lem:pinwheels_use_standard_form}, the pair of tori $\gamma^n$, $F_\theta(\gamma^n)$ are monotone and Hamiltonian isotopic.
	Hence their Lagrangian Floer homology is unobstructed.
	Since they are displaceable, their Floer homology vanishes, for any choice of coefficients.
	If $\gamma^n \cap F_\theta(\gamma^n) = \Delta(\gamma)$, then the tori intersect cleanly by \Cref{cor:clean}, so their pearl complex is defined.	
	However, when $n$ is prime, these facts contradict \Cref{prop:pinwheel_homology_does_not_vanish}.
	It follows that the assumption of \Cref{prop:pinwheel_homology_does_not_vanish} does not hold: $\gamma^n \cap F_\theta(\gamma^n) \ne \Delta(\gamma)$.
	By \Cref{cor:off-diagonal_intersections_are_kewl}, it follows that there exists a degree-($n-1$) inscription of $Q_\theta$ into $\gamma$, as desired.
	If instead $n$ is composite, then let $\ell$ denote a prime factor.
	Letting $p_1(z) = z^{n / \ell}$, we see that $p_1(Q_\theta)$ is the $\theta (n / \ell)$-pinwheel on $2 \ell$ vertices.
	By the case we have just established, there exists a polynomial $p_2 \in \bC[z]_{\ell - 1}^*$ such that $p_2 ( p_1 (Q_\theta)) \subset \gamma$.
	Thus, $p_2 \circ p_1$ gives the desired inscription in this case.
\end{proof}

\noindent
As the proof makes apparent, \Cref{prop:pinwheel} is of greatest interest when $n$ is prime, as the case that $n$ is composite corresponds to a cyclically reducible pinwheel.

The proof of Proposition \ref{prop:pinwheel_homology_does_not_vanish} will proceed as follows.
We shall see that the pearl chain complex of the pair of pinwheel tori exhibits a symmetry:
the relevant strips joining two points on $\Delta(\gamma)$ come in orbits of size $n$.
If there are no intersection points between the tori away from the thin diagonal, then the Floer homology group with coefficients in $\bZ/n$ has total dimension two.

\subsection{Pearl homology: review.}
\label{sec:pearlreview}
Let $L_0$ and $L_1$ denote a pair of compact, spin, monotone Lagrangians in $(\bC^n,\omega)$.
Suppose either that both Lagrangians have minimum Maslov number $\ge 3$ and the same monotonicity constant, or else that they are Hamiltonian isotopic.
In both cases, the pair $(L_0,L_1)$ is unobstructed: their Floer chain complex is defined, following the choice of some auxiliary data, and its homology depends only on the Hamiltonian isotopy classes of the Lagrangians.

Suppose that $L_0$ and $L_1$ intersect cleanly.
We review the pearl complex of the pair $(L_0,L_1)$, following the construction in Schm\"aschke's thesis \cite{schmaschke}.
Its definition requires some auxiliary data:
\begin{itemize}
\item
a choice $J \in \cJ(\bC^n,\omega)$ of smooth, $\omega$-compatible, almost-complex structure;
\item
a Morse function $f : L_0 \cap L_1 \to \bR$;
\item
a metric $g$ on $L_0 \cap L_1$ that makes $(f,g)$ a Morse-Smale pair; and
\item
a field $\bF$.
\end{itemize}

The pearl complex $CF(L_0,L_1;J,f,g,\bF)$ is generated as an $\bF$-vector space by the finite set of critical points $\Crit(f)$.
The choice of field is typically taken to be a Novikov field, as in \cite{schmaschke}.
However, any choice will suffice for the construction in the cases under consideration, as we will explain in this section.
We will vary the choice of $\bF$ in the next section.

The differential on the pearl complex counts rigid pearl trajectories, as follows.
A pearl trajectory from $x \in \Crit(f)$ to $y \in \Crit(f)$ consists of a finite, alternating sequence of Morse trajectories of $(f,g)$ (the `string' of a pearl necklace) and finite energy, non-constant $J$-holomorphic strips
\[ u\colon \bR \times [0,1] \longrightarrow \bC^n \colon (s,t) \longmapsto u(s,t) \]
with boundary on $(L_0,L_1)$ (the `pearls' of a pearl necklace).
Each trajectory is defined on a closed interval.
The initial endpoint of the first interval is $-\infty$, and
the first trajectory limits to $x$ at $-\infty$; while the terminal endpoint of the final interval is $+\infty$, and the last
trajectory limits to $y$ at $+ \infty$.
Every other endpoint is a finite value.
They glue up by the rule that each $J$-holomorphic strip limits as $s \to -\infty$ to the terminal point of the Morse trajectory which precedes it and limits as $s \to +\infty$ to the initial point of the Morse trajectory which follows it in sequence.
Two special cases of pearl trajectories appear below in \Cref{lem:pearls}.

Each pearl trajectory has an expected dimension determined by the Morse indices of $x$ and $y$ and the Maslov-Viterbo indices of the strips appearing in it.
A rigid pearl trajectory is one for which the expected dimension, after dividing out by $\bR$-reparametrization, is 0.
A pair of pearl trajectories from $x$ to $y$ of the same expected dimension differ, up to homotopy, by a pair of classes, one in $\pi_2(M,L_0)$, one in $\pi_2(M,L_1)$, whose Maslov indices sum to 0.
Because $L_0$ and $L_1$ are monotone with the same monotonicity constant, the trajectories therefore have the same symplectic area.
There exists a comeager subset $\cJ_\mathrm{reg} \subset \cJ(\bC^n,\omega)$ of regular almost complex structures such that for all $J \in \cJ_\mathrm{reg}$, and for every non-constant $J$-holomorphic strip $u$ with bounded energy and satisfying the Lagrangian boundary conditions, the linearization $D_{\overline{\del}_J, u}$ surjects.
Since $L_0$ are $L_1$ are spin, it follows that the space of rigid pearl trajectories $\cM(x,y)$ is an oriented 0-manifold, for all $x,y \in \Crit(f)$.
Since the rigid pearl trajectories from $x$ to $y$ all have the same symplectic area, Gromov-Floer compactness ensures that $\cM(x,y)$ is compact.
Fix one such $J$, and let $\del(x,y) \in \bF$ denote the signed count of rigid pearl trajectories from $x$ to $y$.
Let $\del$ denote the endomorphism of $CF(L_0,L_1;J,f,g,\bF)$ defined on generators by the rule $\del(x) = \sum_y \del(x,y) y$.
Then $\del$ is a differential --- $\del^2 = 0$ --- and the resulting homology group computes $HF(L_0,L_1;\bF)$.

Schm\"aschke proves the preceding result under the assumption that the minimum Maslov number of the Lagrangians is $\ge 3$.
The proof that $\del^2=0$ uses the fact that there is no disk bubbling in this case.
He also works over a Novikov field, adjusting $\del(x,y)$ by a monomial depending on the symplectic area.
However, this is only necessary to do if the Lagrangians have different monotonicity constants, in which case $\cM(x,y)$ may not be compact.
The proof also goes through in the case in which the minimum Maslov number is 2 and the Lagrangians are Hamiltonian isotopic, using the fact that Maslov index 2 disk bubbles occur in canceling pairs \cite{OhI,OhII}.

\subsection{Pearl homology of pinwheel tori.}
\label{sec:pearlpinwheel}
Now let $L_0 = \gamma^n$ and $L_1 = F_\theta(L_0)$.
Suppose for a contradiction that $L_0 \cap L_1 = \Delta(\gamma)$.
We aim towards the proof of \Cref{prop:pinwheel_homology_does_not_vanish}.
Granted that, the contradiction is resolved in the proof of \Cref{prop:pinwheel} above.

Define a Morse function $f\colon \Delta(\gamma) \rightarrow \bR$ with a single index-1 critical point $x$ and a single index-0 critical point $y$.
Recall that the cyclic group $\bZ/n \subset \mathrm{Symp}(\bC^n,\omega)$ acts by cyclic permutations of the coordinates, and that $L_0$ and $L_1$ are $\bZ/n$-equivariant (\Cref{lem:pinwheels_use_standard_form}).
A {\em decent neighborhood} $U$ for the pair $(L_0,L_1)$ is a $\bZ/n$-equivariant
tubular neighborhood of the thin diagonal $\Delta(\bC)$
 with the property that $\overline{U} \cap L_0$ and $\overline{U} \cap L_1$ are tubular neighborhoods of $\Delta(\gamma)$ within the respective tori.  The existence of decent neighborhoods follows from the verification that the intersection $\Delta(\gamma)$ is clean (Corollary \ref{cor:clean}).

\begin{lem}
	\label{lem:pearls}
	Every rigid pearl trajectory for $(L_0,L_1)$ is either \textup{(a)} a single Morse trajectory from $x$ to $y$ or else \textup{(b)} a constant Morse trajectory at $y$, followed by a $J$-holomorphic strip of Maslov-Viterbo index 2 from $y$ to $x$, and finished by a constant Morse trajectory at $x$.
\end{lem}

\begin{proof}
	The proof is based on the numerology of the expected dimensions of pearl trajectories \cite[Lemma 10.2.1]{schmaschke}.
	It implies that for a rigid pearl trajectory from $x$ to $y$, the Maslov-Viterbo indices of the strips must sum to 0, while for a rigid pearl trajectory from $y$ to $x$, they must sum to 2.
	
	Choose a pair of points $z_\pm \in \Delta(\gamma)$, and consider pairs of paths given by continuous maps $u \colon \bR \times \{0,1\} \rightarrow \bC^2 \colon (s,t) \mapsto u(s,t)$ with $u(\bR \times \{ i \}) \subseteq L_i$ for $i=0,1$, and limiting to $z_\pm$ as $s \rightarrow \pm \infty$, respectively.  Such pairs of paths fall into homotopy classes.
	The symplectic area of a strip extending such a pair of paths is independent of the extension, and it is constant over the homotopy class of the pair of paths.
	In addition, the Maslov-Viterbo index of the strip is equal to the Robbin-Salamon index of the pair of paths comprising its boundary \cite{robbinandsalamon}.
	
	Let $v$ be such a pair of paths with $v(s,0) = v(s,1)$ for all $s \in \bR$.
	Then $v$ runs along the clean intersection $L_0 \cap L_1 = \Delta(\gamma)$, so $T L_0$ and $T L_1$ intersect in a $1$-dimensional subspace at each point along it.
	 It follows that the Robbin-Salamon index of $v$ is $0$ by the (ZERO) axiom \cite{robbinandsalamon}.  
	Every other pair of paths is homotopic to $v$ after possibly taking a connected sum with appropriate representatives of classes in $H_1(L_0)$ and $H_1(L_1)$.
	The homotopy does not affect the Robbin-Salamon index (HOMOTOPY), and the connected sum changes it by the Maslov indices of the classes (CATENATION).
	Since $L_0$ and $L_1$ are orientable, these Maslov indices are even.
	This shows that the Maslov-Viterbo index of a strip $u$ limiting to $z_\pm$ is even.
	Moreover, $L_0$ and $L_1$ are monotone with monotonicity constant $\lambda$ equal to half the area bounded by $\gamma \subset \bC$.
	This shows that the Maslov-Viterbo index of $u$ is equal to the symplectic area of $u$ divided by $\lambda$, hence is positive.
	In particular, if $u$ is non-constant, then its index equals 2.

	It follows that a rigid pearl trajectory from $x$ to $y$ cannot involve any strips, while a rigid pearl trajectory from $y$ to $x$ involves a unique strip.
	This implies that rigid pearl trajectories take the stated form.
\end{proof}

\begin{lem}
	\label{lem:decency}
	If $u$ is a $J$-holomorphic strip in a rigid pearl trajectory and $U$ is a decent neighborhood of $(L_0,L_1)$, then the image of $u$ exits $\overline{U}$.
\end{lem}

\begin{proof}
	We establish the contrapositive.
	Suppose that $u$ is a $J$-holomorphic strip joining points $z_\pm \in L_0 \cap L_1$ with $\im(u) \subset \overline{U}$.
	Because $\overline{U} \cap L_0$ and $\overline{U} \cap L_1$ are closed tubular neighborhoods of $L_0 \cap L_1$ within the respective Lagrangians, the boundary arcs of $u$ can be homotoped rel endpoints into $L_0 \cap L_1$.
	Summing the new arcs with copies of $L_0 \cap L_1$ brings them into a pair of identical simple arcs, which has Robbin-Salamon index 0, as in the proof of \Cref{lem:pearls}.
	The loop $L_0 \cap L_1$ has Maslov index $2 n$ in each of $L_0$ and $L_1$.
	It follows that $\mu(u) \equiv 0 \pmod {2 n}$.
	By \Cref{lem:pearls}, it follows that $u$ cannot contribute to a rigid pearl trajectory of $(L_0,L_1)$.
\end{proof}

Establishing the existence of a regular $\bZ/n$-equivariant almost complex structure will be the final ingredient in establishing Proposition \ref{prop:pinwheel_homology_does_not_vanish}.

\begin{lem}
	\label{lem:existence_regular_equivariant_J}
	Suppose that $n$ is prime, and let $U$ denote a decent neighborhood of $(L_0,L_1)$.
	Then there exists a regular almost complex structure $J$ for $(L_0,L_1)$ such that (a) $J$ is $\bZ/n$-equivariant on $\bC^n \setminus U$ and (b) $J | U = J_\mathrm{std} | U$.
\end{lem}

\begin{proof}
We wish to follow the proof of \cite[Theorem 3.1.2]{slimmcduffsalamon}.  The extra complication in our situation is that we are working with equivariant almost-complex structures, and this means we need to pay more attention to the universal moduli space.
We set this up as follows.

For any $\ell \ge 0$, let $\cJ^\ell$ denote the space of $C^\ell$ $\omega$-compatible almost complex structures on $\bC^n$.
The action by $\bZ/n$ on $\bC^n \setminus U$ is free, since $n$ is prime and $\Delta(\bC) \subset U$.
Therefore, the quotient $M = (\bC^n \setminus U) / \bZ/n$ is a manifold with boundary, and it carries a quotient symplectic form $\overline{\omega}$.
Let $\overline{\cJ}_U^\ell$ denote the space of $C^\ell$ $\overline{\omega}$-compatible almost complex structures on $M$ which agree on $\del M$ with the image of $J_\mathrm{std}$ under the quotient map.
This space is a Banach manifold whose tangent space to $\overline{J} \in \overline{\cJ}_U^\ell$ is the space of $C^\ell$-sections of the bundle $\mathrm{End}(M,\overline{J},\overline{\omega})$ that vanish on $\del M$ (see \cite[Section 3.4]{slimmcduffsalamon} for a description of this bundle).
We can identify $\overline{\cJ}_U^\ell$ with the subspace $\cJ_U^\ell \subset \cJ^\ell$ of $\bZ/n$-equivariant almost complex structures which agree with $J_\mathrm{std}$ on $\overline{U}$.
Thus, $\cJ_U^\ell$ has the structure of a Banach manifold.

Now consider the universal moduli space $\cM^\ell$ consisting of pairs $(J,u)$, where $J \in \cJ^\ell_U$ and $u$ is a finite energy, $J$-holomorphic strip which satisfies the Lagrangian boundary conditions and has Maslov-Viterbo index 2.
We wish to establish that the space $\cM^\ell$ is a Banach manifold, which is an analogue of \cite[Proposition 3.4.1]{slimmcduffsalamon}, straightforwardly adapted to the setting of Lagrangian Floer homology.
Here we must use Aronszajn's theorem, as indicated in the footnote to that proof; see as well \cite[Proposition 3.2.1 and Remark 3.2.3]{bigmcduffsalamon}. Once we have this, then the proof of \cite[Theorem 3.1.2]{slimmcduffsalamon} carries over to our setting and we will have established Lemma \ref{lem:existence_regular_equivariant_J}.

The one addition to \cite[Proposition 3.4.1]{slimmcduffsalamon} that must be made is that, because we use $\bZ/n$-equivariant almost complex structures which are fixed on $\overline{U}$, we must argue that for each $(J,u) \in \cM^\ell$, there exists a point $z_0 \in \bR \times [0,1]$ such that
\[
u(z_0) \in \bC^n \setminus \overline{U}, \quad du(z_0) \ne 0, \quad u^{-1}(\bZ/n \cdot u(z_0)) = \{z_0\}.
\]
Granted this, the proof carries through as in \cite{slimmcduffsalamon}.

Given such a pair $(J,u)$, consider the mapping
\[
v : \left( \bR \times [0,1] \right) \times \bZ/n \to \bC^n, \quad v(z,\sigma) = \sigma \cdot u(z).
\]
We record a few basic observations about this map.
Its domain consists of $n$ pairwise disjoint strips.
It is $J$-holomorphic, since $u$ is $J$-holomorphic and $J$ is $\bZ/n$-equivariant.
Its restriction to each component of its domain has Maslov-Viterbo index 2.
Its image is not contained in $\overline{U}$, since $\im(u) \subset \im(v)$, and $\im(u) \not\subset \overline{U}$ by \Cref{lem:decency}.
Lastly, it has finite energy: that is because the map $u$ has finite energy, since $L_0$ and $L_1$ intersect cleanly.

The desired point $z_0$ for the pair $(J,u)$ exists provided we can show that there exists a point $(z_0,\sigma)$ in the domain of $v$ such that
\[
v(z_0,\sigma) \in \bC^n \setminus  \overline{U}, \quad dv(z_0,\sigma) \ne 0, \quad v^{-1}(v(z_0,\sigma)) = \{(z_0,\sigma)\}.
\]
The second and third conditions mean that $(z_0,\sigma)$ is an injective point of $v$.
By \cite[Proposition 2.3.1]{slimmcduffsalamon}, either (a) the set of injective points of $v$ is a dense open set of its domain, or else (b) $v$ is multiply-covered: : $v = w \circ \phi$, where $\phi$ is a non-trivial branched covering map and $w$ is $J$-holomorphic.
In case (a), we may therefore locate an injective point $(z_0,\sigma)$ of $v$ whose image under $v$ is not contained in $\overline{U}$, since $\im(v) \not\subset \overline{U}$.
The proof now reduces to showing that case (b) cannot occur.

Suppose for a contradiction that $v$ is multiply-covered.
It follows that there exists a component $(\bR \times [0,1]) \times \{\sigma\}$ of the domain of $v$ to which $\phi$ restricts as a non-trivial branched covering map.
The only space which a strip non-trivially branch-covers is a disk with a single boundary puncture: $\bD \setminus \{1\}$.
The boundary conditions on $v$ imply that $w_\sigma( \del (\bD \setminus \{1\}) ) \subset L_0 \cap L_1$.
It follows that $v_\sigma$ maps the boundary of the strip into $L_0 \cap L_1$.
This implies that its Maslov-Viterbo index is a multiple of $2 n$ as in the proof of Lemma \ref{lem:decency}, hence not equal to $2$, a contradiction.
\end{proof}

\begin{proof}[Proof of Proposition \ref{prop:pinwheel_homology_does_not_vanish}.]
Fix a Morse function $f$ and metric $g$ as above, a decent neighborhood $U$, and an almost-complex structure $J$ as in \Cref{lem:existence_regular_equivariant_J}.
A rigid pearl trajectory of $(L_0,L_1)$ with image contained in $\overline{U}$ is a Morse trajectory of $(f,g)$ by Lemma \ref{lem:decency}, and the two have opposite signs.
A rigid pearl trajectory with image exiting $\overline{U}$ occurs in an orbit of such rigid pearl trajectories of size $n$ under the $\bZ/n$-action, since $J$ is $\bZ/n$-equivariant by \Cref{lem:existence_regular_equivariant_J} and the pair $(L_0,L_1)$ is $\bZ/n$-equivariant by \Cref{lem:pinwheels_use_standard_form}.
Furthermore, after choosing for each of $L_0, L_1$ one of the two $\bZ/n$-equivariant spin structures, we see that all of the rigid pearl trajectories within such an orbit have the same sign.
It follows that the differential on $CF(L_0,L_1;J,f,g,\bZ/n)$ vanishes.
Thus, the homology $HF(L_0,L_1;\bZ/n)$ is two-dimensional, generated by the classes of $x$ and $y$.
\end{proof}

\newpage

\bibliographystyle{amsplain}
\bibliography{References./works-cited.bib}

\providecommand{\bysame}{\leavevmode\hbox to3em{\hrulefill}\thinspace}
\providecommand{\MR}{\relax\ifhmode\unskip\space\fi MR }
\providecommand{\MRhref}[2]{%
  \href{http://www.ams.org/mathscinet-getitem?mr=#1}{#2}
}
\providecommand{\href}[2]{#2}
\begin{thebibliography}{10}

\bibitem{audin1988}
Mich\`ele Audin, \emph{Fibr\'es normaux d'immersions en dimension double,
  points doubles d'immersions lagrangiennes et plongements totalement r\'eels},
  Comment. Math. Helv. \textbf{63} (1988), no.~4, 593--623.

\bibitem{bezout}
Etienne B\'ezout, \emph{General theory of algebraic equations}, 1779.

\bibitem{cassini}
Jean-Dominique Cassini, \emph{De l'origine et du progr\'es de l'astronomie et
  de son usage dans la g\'eographie et dans la navigation}, 1693.

\bibitem{cieliebakmohnke2018}
Kai Cieliebak and Klaus Mohnke, \emph{Punctured holomorphic curves and
  {L}agrangian embeddings}, Invent. Math. \textbf{212} (2018), no.~1, 213--295.

\bibitem{eems2013}
Tobias Ekholm, Yakov Eliashberg, Emmy Murphy, and Ivan Smith,
  \emph{Constructing exact {L}agrangian immersions with few double points},
  Geom. Funct. Anal. \textbf{23} (2013), no.~6, 1772--1803.

\bibitem{fukaya_langrangian_submanifolds}
Kenji Fukaya, \emph{Application of {F}loer homology of {L}angrangian
  submanifolds to symplectic topology}, Morse theoretic methods in nonlinear
  analysis and in symplectic topology, NATO Sci. Ser. II Math. Phys. Chem.,
  vol. 217, Springer, Dordrecht, 2006, pp.~231--276.

\bibitem{greenelobb2}
Joshua~Evan Greene and Andrew Lobb, \emph{Cyclic quadrilaterals and smooth
  {J}ordan curves}, Invent. Math. \textbf{234} (2023), no.~3, 931--935.

\bibitem{greenelobb4}
\bysame, \emph{Floer homology and square pegs}, {\tt arXiv:2404.05179} (2024).

\bibitem{grunbaum}
Branko Gr\"unbaum, \emph{Arrangements and spreads}, Conference Board of the
  Mathematical Sciences Regional Conference Series in Mathematics, vol. No. 10,
  American Mathematical Society, Providence, RI, 1972.

\bibitem{irie}
Kei Irie, \emph{Chain level loop bracket and pseudo-holomorphic disks}, J.
  Topol. \textbf{13} (2020), no.~2, 870--938.

\bibitem{slimmcduffsalamon}
Dusa McDuff and Dietmar Salamon, \emph{{$J$}-holomorphic curves and quantum
  cohomology}, American Mathematical Society, Providence, RI, 1994.

\bibitem{bigmcduffsalamon}
\bysame, \emph{{$J$}-holomorphic curves and symplectic topology}, second ed.,
  American Mathematical Society Colloquium Publications, vol.~52, American
  Mathematical Society, Providence, RI, 2012.

\bibitem{newton}
Isaac Newton, \emph{Principia mathematica}, 1687.

\bibitem{OhI}
Yong-Geun Oh, \emph{Floer cohomology of {L}agrangian intersections and
  pseudo-holomorphic disks. {I}}, Comm. Pure Appl. Math. \textbf{46} (1993),
  no.~7, 949--993. \MR{1223659}

\bibitem{OhII}
\bysame, \emph{Floer cohomology of {L}agrangian intersections and
  pseudo-holomorphic disks. {II}. {$({\bf C}{\rm P}^n,{\bf R}{\rm P}^n)$}},
  Comm. Pure Appl. Math. \textbf{46} (1993), no.~7, 995--1012. \MR{1223660}

\bibitem{polto1991}
Leonid Polterovich, \emph{The {M}aslov class of the {L}agrange surfaces and
  {G}romov's pseudo-holomorphic curves}, Trans. Amer. Math. Soc. \textbf{325}
  (1991), no.~1, 241--248.

\bibitem{polterovich_smoothing}
\bysame, \emph{The surgery of {L}agrange submanifolds}, Geom. Funct. Anal.
  \textbf{1} (1991), no.~2, 198--210. \MR{1097259}

\bibitem{pozniak}
Marcin Po{\'z}niak, \emph{Floer homology, {N}ovikov rings and clean
  intersections}, Ph.D. thesis, University of {W}arwick, 1994.

\bibitem{robbinandsalamon}
Joel Robbin and Dietmar Salamon, \emph{The {M}aslov index for paths}, Topology
  \textbf{32} (1993), no.~4, 827--844. \MR{1241874}

\bibitem{schmaschke}
Felix Schm{\"a}schke, \emph{Abelianization and {F}loer homology of
  {L}agrangians in clean intersection}, Ph.D. thesis, Universit{\"a}t
  {L}eipzig, 2017.

\bibitem{toeplitz1911}
Otto Toeplitz, \emph{Ueber einige {A}ufgaben der {A}nalysis situs},
  Verhandlungen der {S}chweizerischen {N}aturforschenden {G}esellschaft (1911),
  no.~4, 197.

\bibitem{viterbo1990}
Claude Viterbo, \emph{A new obstruction to embedding {L}agrangian tori},
  Invent. Math. \textbf{100} (1990), no.~2, 301--320.

\bibitem{vreziv}
Sini\v sa~T. Vre\'cica and Rade~T. \v{Z}ivaljevi\'c,
  \emph{Fulton-{M}ac{P}herson compactification, cyclohedra, and the polygonal
  pegs problem}, Israel J. Math. \textbf{184} (2011), 221--249.

\end{thebibliography}
\end{document}